\documentclass[12pt, notitlepage]{amsart}
\usepackage{latexsym, amsfonts, amsmath, amssymb, amsthm}

\newtheorem{theorem}{Theorem}[section]
\newtheorem{lemma}[theorem]{Lemma}
\newtheorem{proposition}[theorem]{Proposition}

\newtheorem{corollary}[theorem]{Corollary}

\newtheorem{remark}[theorem]{Remark}

\newcommand{\Z}{\mathbb{Z}}
\newcommand{\R}{\mathbb{R}}

\newcommand{\F}{\mathbb{F}}

\usepackage{graphicx}
\usepackage{mathrsfs}

\begin{document}

\title{Character sums over unions of intervals}
\author{Xuancheng Shao}

\maketitle

\begin{abstract}
Let $q$ be a cube-free positive integer and $\chi\pmod q$ be a
non-principal Dirichlet character. Our main result is a Burgess-type
estimate for $\sum_{n\in A}\chi(n)$, where $A\subset [1,q]$ is the
union of $s$ disjoint intervals $I_1,\cdots,I_s$. We obtain a
nontrivial estimate for the character sum over $A$ whenever
$|A|s^{-1/2}>q^{1/4+\epsilon}$ and each interval $I_j$ ($1\leq j\leq
s$) has length $|I_j|>q^{\epsilon}$ for any $\epsilon>0$. This
follows from an improvement of a mean value Burgess-type estimate
studied by Heath-Brown \cite{heath-brown}.
\end{abstract}

\section{Introduction}

Let $\chi\pmod q$ be a primitive Dirichlet character. Write
\[ S(N;H)=\sum_{N<n\leq N+H}\chi(n). \]
The well-known estimates of Burgess
\cite{burgess1,burgess2,burgess3} say that for any positive integer
$r\geq 2$ and any positive real $\epsilon>0$,
\[ S(N;H)\ll_{\epsilon,r}H^{1-1/r}q^{(r+1)/(4r^2)+\epsilon} \]
uniformly in $N$, providing that either $q$ is cube-free or $r\leq
3$. In particular, this gives a nontrivial estimate for $S(N;H)$ as
long as $H>q^{1/4+\epsilon}$ for any $\epsilon>0$. Under the
assumption of the Generalized Riemann Hypothesis (GRH), nontrivial
estimates for $S(N;H)$ can be obtained in the much wider region
$H>q^{\epsilon}$ for any $\epsilon>0$. However, Burgess's bound
remains the best unconditional result for around 50 years.

Our main purpose is to prove the following mean value Burgess-type
estimate, which includes the Burgess bound as a special case.

\begin{theorem}\label{thm:main}
Let $r$ be a positive integer and $\epsilon>0$ be real. Suppose that
$\chi\pmod q$ is a primitive Dirichlet character, and let $H\leq q$
be a positive integer. Let $0\leq N_1<N_2<\cdots<N_J<q$ be integers
with the spacing condition
\[ N_{j+1}-N_j\geq H. \]
Then \begin{equation}\label{2r} \sum_{j=1}^J\max_{h\leq
H}|S(N_j;h)|^{2r}\ll_{\epsilon,r} q^{1/2+1/(2r)+\epsilon}H^{2r-2}
\end{equation}
under any one of the following three conditions
\begin{enumerate}
\item[(a)] $r=1$;
\item[(b)] $r\leq 3$ and $H\geq q^{1/(2r)}$;
\item[(c)] $q$ is cube-free and $H\geq q^{1/(2r)}$.
\end{enumerate}
\end{theorem}

This question has been investigated recently by Heath-Brown
\cite{heath-brown}, where the following inequality is obtained:
\begin{equation}\label{3r}
 \sum_{j=1}^J\max_{h\leq
H}|S(N_j;h)|^{3r}\ll_{\epsilon,r} q^{3/4+3/(4r)+\epsilon}H^{3r-3},
\end{equation}
under the same assumptions as those in Theorem \ref{thm:main}.
Heath-Brown attained $2r$ when $r=1$, and asked the question of
whether the exponent $3r$ in \eqref{3r} can be replaced by $2r$.
Theorem \ref{thm:main} is then an affirmative answer to this
question.

This improvement from $3r$ to $2r$ has a major consequence. The left
sides of the inequalities \eqref{3r} and \eqref{2r} have trivial
bounds $JH^{3r}$ and $JH^{2r}$, respectively. Some simple algebra
then reveals that \eqref{3r} is nontrivial for large $r$ when
$HJ^{1/3}>q^{1/4+\epsilon}$ for any $\epsilon>0$, whereas \eqref{2r}
is nontrivial for large $r$ when $HJ^{1/2}>q^{1/4+\epsilon}$ for any
$\epsilon>0$.

Theorem \ref{thm:main} easily implies a Burgess-type bound for
character sums over unions of intervals.

\begin{corollary}\label{cor:main}
Let $\epsilon>0$ be real. Suppose that $q$ is cube-free and let
$\chi\pmod q$ be a primitive Dirichlet character. Let $A\subset
[1,q]$ be a union of $s$ disjoint intervals $I_1,\cdots,I_s$, each
of which has length at least $q^{\epsilon}$. Suppose that
$|A|s^{-1/2}>q^{1/4+\epsilon}$. Then there exists
$\delta=\delta(\epsilon)>0$ such that
\begin{equation}\label{0} \sum_{n\in A}\chi(n)\ll_{\epsilon} |A|q^{-\delta}.
\end{equation}
\end{corollary}

This problem of getting nontrivial bounds for character sums over
unions of intervals has been considered before. Friedlander and
Iwaniec \cite{friedlander} proved \eqref{0} but with an additional
assumption that $A$ is contained in a relatively short interval. The
result of Chang \cite{chang} also gives \eqref{0} under a stronger
hypothesis on $|A|$ and $s$. Chang's method is different, though,
and can be also used to treat character sums over generalized
arithmetic progressions. Finally, note that Heath-Brown's result
\eqref{3r} would lead to $\eqref{0}$ assuming
$|A|s^{-2/3}>q^{1/4+\epsilon}$. Heuristically, square-root
cancelation for the character sum is expected over the $s$
intervals, and thus the set $A$ can be thought of having an
effective length $|A|s^{-1/2}$. This heuristic suggests that
Corollary \ref{cor:main} might be best possible (at least
unconditionally) under current techniques, as any improvement may
also lead to an improvement of Burgess's inequality.

The proof of Theorem \ref{thm:main} follows the initial argument of
Heath-Brown \cite{heath-brown}, where a variant of Burgess's method
was used to convert the problem to a diophantine problem. We refer
the reader to Section 3 of \cite{heath-brown} for this process. The
improvement is a consequence of a better bound for this diophantine
problem, as described below.

\begin{proposition}\label{prop:main}
Let $\ell$ be a prime and $S$ be a subset of $\F_{\ell}$. Let $n$ be
a positive integer. Denote by $N(\ell,S,n)$ the number of solutions
to the congruence $as-bt\equiv c\pmod\ell$ with $1\leq a,b\leq n$,
$|c|\leq n$, and $s,t\in S$. Then \begin{equation}\label{N}
N(\ell,S,n)\ll_{\epsilon} \ell^{-1}n^3|S|^2+\ell^{\epsilon}n^2|S|
\end{equation}
for any $\epsilon>0$.
\end{proposition}

In Section 4 of \cite{heath-brown}, Heath-Brown obtained the bound
\[ N(\ell,S,n)\ll \ell^{-1}n^3|S|^2+|S|^2+n^2|S|^{4/3}\log\ell \]
using the theory of lattices. The proof of Proposition
\ref{prop:main} uses Fourier analysis instead. Note that, in the
bound \eqref{N}, the first term $\ell^{-1}n^3|S|^2$ is the expected
main term, and the second term $\ell^{\epsilon}n^2|S|$ is also sharp
(apart from the $\ell^{\epsilon}$ factor). This can be seen by
taking $S$ to be a set of consecutive integers starting from $1$.

Throughout the paper the implied constants in the $O()$ and $\ll$
notations are always allowed to depend on $r$ and $\epsilon$. The
parameter $\epsilon$ appearing in different places are allowed to
differ. For a prime $\ell$ and a function
$f:\F_{\ell}\rightarrow\R$, we define its Fourier transform to be
\[ \hat{f}(r)=\sum_{x\in\F_{\ell}}f(x)e_{\ell}(xr) \]
for $r\in\F_{\ell}$, where $e_{\ell}(xr)=\exp(2\pi ixr/\ell)$. For
an $L^1$ function $f:\R\rightarrow\R$, we define its Fourier
transform to be
\[ \hat{f}(y)=\int_{\R}f(x)e(xy)dx, \]
where $e(xy)=\exp(2\pi ixy)$.

The rest of this article is organized as follows. Heath-Brown's
argument, which converts the problem of bounding character sums to
the problem of bounding the number of solutions to a certain
congruence equation, is summarized in Section \ref{sec:prep}.
Proposition \ref{prop:main} is proved in Section \ref{sec:prop}.
Theorem \ref{thm:main} is then deduced from Proposition
\ref{prop:main} in Section \ref{sec:thm}, and finally Corollary
\ref{cor:main} is obtained in Section \ref{sec:cor}.


\section{Preparations}\label{sec:prep}

In this section, we summarize Heath-Brown's work, setting up the
foundation for the proof of Theorem \ref{thm:main}.

\begin{proposition}\label{prop:hb}
Let $r\geq 2$ be a positive integer and $\epsilon>0$ be real.
Suppose that $\chi\pmod q$ is a primitive Dirichlet character, and
let $H\in (q^{1/(2r)},q]$ be a positive integer. Assume that either
$q$ is cube-free or $r\leq 3$. Let $0\leq N_1<N_2<\cdots<N_J<q$ be
integers with the spacing condition
\[ N_{j+1}-N_j\geq H. \]
Let $\ell\in (q/H,2q/H]$ be a prime. Let $P$ be a parameter
satisfying $2Hq^{-1/(2r)}\leq P\ll Hq^{-1/(2r)}$. Then for some
subset $S\subset\F_{\ell}$ with $|S|=J$,
\[ \sum_{j=1}^J\max_{h\leq H}|S(N_j;h)|^r\ll
q^{1/4+3/(4r)+\epsilon}H^{r-2}(PJ^{1/2}+N(\ell,S,12P)^{1/2}), \]
where $N(\ell,S,12P)$ is defined as in the statement of Proposition
\ref{prop:main}.
\end{proposition}

\begin{proof}
We outline the arguments from \cite{heath-brown}. Using a variant of
Burgess's method Heath-Brown obtained (equation (7) of
\cite{heath-brown})
\begin{equation}\label{hb1}
\sum_{j=1}^J\max_{h\leq H}|S(N_j;h)|^r\ll
q^{1/4+3/(4r)+\epsilon}H^{r-2}\mathcal{M}^{1/2}, \end{equation}
where
$\mathcal{M}$ is the number of tuples $(a_1,a_2,p_1,p_2,N_j,N_k)$
with $p_1,p_2$ primes in $(P,2P]$, $0\leq a_1<p_1$, $0\leq a_2<p_2$,
and
\[ |(N_j-a_1q)/p_1-(N_k-a_2q)/p_2|\leq H/P. \]

In the beginning of Section 4 of \cite{heath-brown}, Heath-Brown
started to bound $\mathcal{M}$ as follows. Split $\mathcal{M}$ as
$\mathcal{M}_1+\mathcal{M}_2$, where $\mathcal{M}_1$ counts
solutions with $p_1=p_2$ and $\mathcal{M}_2$ counts those with
$p_1\neq p_2$. It can be easily seen that (equation (8) of
\cite{heath-brown})
\[ \mathcal{M}_1\ll P^2J. \]
To bound $\mathcal{M}_2$, set $M_j=[N_j\ell/q]$ for each $1\leq
j\leq J$. The spacing condition on $N_j$ implies that
\[ 0\leq M_1<M_2<\cdots<M_J<\ell. \]
Let $S=\{M_1,M_2,\cdots,M_J\}\subset\F_{\ell}$. It is then shown
that (the first displayed equation on page 11 of \cite{heath-brown})
\[ \mathcal{M}_2\leq\sum_{M_j,M_k}\#\{(p_1,p_2,m):|m|\leq
12P,p_2M_j-p_1M_k\equiv m\pmod\ell\}. \]
The right side above is
bounded above by $N(\ell,S,12P)$. Hence \[
\mathcal{M}=\mathcal{M}_1+\mathcal{M}_2\ll P^2J+N(\ell,S,12P). \]
Combining this with \eqref{hb1} completes the proof.
\end{proof}

Heath-Brown's subsequent estimates take advantage of the fact that
$p_1$ and $p_2$ are primes; however, our method is insensitive to
this.

We will also need the following P\'{o}lya-Vinogradov-type mean value
estimate.

\begin{lemma}\label{lem:mean}
Let $r$ be a positive integer and $\epsilon>0$ be real. Suppose that
$\chi\pmod q$ is a primitive Dirichlet character, and let $H\leq q$
be a positive integer. Let $0\leq N_1<N_2<\cdots<N_J<q$ be integers
with the spacing condition
\[ N_{j+1}-N_j\geq H. \]
Then
\[ \sum_{j=1}^J\max_{h\leq H}|S(N_j;h)|^{2}\ll
q(\log q)^2 \] for all $q$, and
\[ \sum_{j=1}^J\max_{h\leq H}|S(N_j;h)|^{2r}\ll
q^{\epsilon}(qH^{r-1}+q^{1/2}H^{2r-1}) \] under the assumption that
either $2\leq r\leq 3$ or $q$ is cube-free.
\end{lemma}

The $r=1$ case is Lemma 4 in \cite{heath-brown}. The proof there,
however, works for arbitrary $r$.


\section{Proof of Proposition \ref{prop:main}}\label{sec:prop}

Recall that $N(\ell,S,n)$ is the number of solutions to
\[ as-bt\equiv c\pmod\ell \]
with $1\leq a,b\leq n$, $|c|\leq n$, and $s,t\in S$. If $\ell\leq
n$, then $N(\ell,S,n)$ can be bounded easily as follows. For any
fixed choice of $a,b,s,t$, there are at most $\lceil n/\ell\rceil$
choices for $c$. Therefore
\[ N(\ell,S,n)\ll \ell^{-1}n^3|S|^2 \]
as desired. Henceforth assume that $\ell>n$.

To facilitate the argument we introduce a smooth cutoff
$\phi:\R\rightarrow\R^+$ satisfying the following properties:
\begin{enumerate}
\item $\phi(x)\geq 0$ for every $x\in\R$ and $\phi(x)$ is bounded
away from $0$ for $x\in [-1,1]$;
\item $\hat{\phi}$ is supported in the interval $[-1/10,1/10]$ and
$|\hat{\phi}(y)|\leq 1$ for every $y\in\R$.
\end{enumerate}
Such a function $\phi$ can be easily constructed, for example, by
taking $\phi(x)=(\sin (tx)/tx)^2$ for some appropriate $t>0$.

We denote by $S$ the characteristic function of the set $S$. Since
$\phi(x)\gg 1$ for $x\in [-1,1]$ and $\phi(x)\geq 0$ for all $x$, we
have the bound
\[ T\leq\sum_{1\leq a,b\leq
n}\sum_{s,t\in\F_{\ell}}\sum_{\substack{c\in\Z \\ as-bt\equiv
c\pmod\ell}}S(s)S(t)\phi(c/n). \] Then by orthogonality of additive
characters,
\[ T\leq\frac{1}{\ell}\sum_{1\leq a,b\leq
n}\sum_{s,t\in\F_{\ell}}\sum_{c\in\Z}S(s)S(t)\phi(c/n)\sum_{r\in\F_{\ell}}e_{\ell}(r(as-bt-c)).
\]
Changing the order of summation we get
\[ T\leq\frac{1}{\ell}\sum_{|r|\leq\ell/2}\sum_{1\leq a,b\leq
n}\hat{S}(ar)\hat{S}(-br)\sum_{c\in\Z}\phi(c/n)e_{\ell}(-cr). \] By
Poisson summation,
\[
\sum_{c\in\Z}\phi(c/n)e_{\ell}(-cr)=n\sum_{k\in\Z}\hat{\phi}\left(n\left(k-\frac{r}{\ell}\right)\right).
\]
Since $\hat{\phi}$ is compactly supported in $[-1/10,1/10]$, the
summand on the right above vanishes unless $k=0$ and
$|r|\leq\ell/5n$. Hence
\[
T\leq\frac{n}{\ell}\sum_{|r|\leq\ell/5n}\sum_{1\leq a,b\leq
n}\hat{S}(ar)\hat{S}(-br)\hat{\phi}(-nr/\ell)\ll\frac{n^3|S|^2}{\ell}+\frac{n}{\ell}\sum_{0<|r|\leq\ell/5n}\sum_{1\leq
a,b\leq n}|\hat{S}(ar)\hat{S}(-br)|
\]
since $|\hat{\phi}(y)|\leq 1$ for all $y$. It follows from the
inequality
\[ 2|\hat{S}(ar)\hat{S}(-br)|\leq |\hat{S}(ar)|^2+|\hat{S}(-br)|^2 \]
that
\[
T\ll\frac{n^3|S|^2}{\ell}+\frac{n^2}{\ell}\sum_{0<|r|\leq\ell/5n}\sum_{1\leq
|a|\leq n}|\hat{S}(ar)|^2. \] Note that for any fixed nonzero
$s\in\F_{\ell}$, there are $O(\ell^{\epsilon})$ ways to write
$s\equiv ar\pmod\ell$ with $1\leq |a|\leq n$ and $0<|r|\leq\ell/5n$.
Hence
\[ \sum_{0<|r|\leq\ell/5n}\sum_{1\leq
|a|\leq
n}|\hat{S}(ar)|^2\ll\ell^{\epsilon}\sum_{s\in\F_{\ell}}|\hat{S}(s)|^2\ll\ell^{1+\epsilon}|S|
\]
by Parseval's identity. This gives the desired bound
\[ T\ll\frac{n^3|S|^2}{\ell}+\ell^{\epsilon}n^2|S|, \]
completing the proof of Proposition \ref{prop:main}.


\section{Proof of Theorem \ref{thm:main}}\label{sec:thm}

It follows immediately from Propositions \ref{prop:hb} and
\ref{prop:main} that
\[ \sum_{j=1}^J\max_{h\leq H}|S(N_j;h)|^r\ll
q^{1/4+3/(4r)+\epsilon}H^{r-2}\left(PJ^{1/2}+\frac{P^{3/2}J}{\ell^{1/2}}\right).
\]
Recall that $P\sim Hq^{-1/2r}$ and $\ell\sim q/H$. Hence
\begin{equation}\label{4} \sum_{j=1}^J\max_{h\leq H}|S(N_j;h)|^r\ll
q^{1/4+1/(4r)+\epsilon}H^{r-1}J^{1/2}+q^{-1/4+\epsilon}H^rJ.
\end{equation} Before proving Theorem \ref{thm:main} we
remark that the bound above is already nontrivial when
$HJ^{1/2}>q^{1/4+\epsilon}$, and is thus sufficient to deduce
Corollary \ref{cor:main}. However a little more work needs be done
to get Theorem \ref{thm:main} as stated. The arguments here are
analogous to those in Section 5 of \cite{heath-brown}.

We use induction on $r$. The $r=1$ case is simply Lemma
\ref{lem:mean}. Now assume that $r\geq 2$. By a dyadic subdivision,
we may assume that
\[ \max_{h\leq H}|S(N_j;h)|\sim V \]
for some $V$. Hence from \eqref{4},
\[ JV^r\ll q^{1/4+1/(4r)+\epsilon}H^{r-1}J^{1/2}+q^{-1/4}H^rJ. \]
Consider two cases depending on which term on the right side above
dominates. If the first term dominates, then, upon moving $J^{1/2}$
to the left, we get
\[  JV^{2r}\ll q^{1/2+1/(2r)+\epsilon}H^{2r-2}, \]
as desired. If the second term dominates, then
\[ JV^r\ll q^{-1/4+\epsilon}H^rJ, \]
and thus $V\ll q^{-1/(4r)+\epsilon/r}H$. Divide further into two
cases.

If $H>q^{1/2(r-1)}$, then we may use the induction hypothesis with
$r-1$ to deduce that
\[ JV^{2(r-1)}\ll q^{1/2+1/2(r-1)+\epsilon}H^{2r-4}. \]
Hence
\[ JV^{2r}\ll q^{1/2+1/2(r-1)+\epsilon}H^{2r-4}\cdot
q^{-1/(2r)+2\epsilon/r}H^2. \] The desired bound follows from the
inequality $1/2(r-1)-1/(2r)\leq 1/(2r)$ when $r\geq 2$.

If $H\leq q^{1/2(r-1)}$, then we use the conclusion of Lemma
\ref{lem:mean},
\[ \sum_{j=1}^J\max_{h\leq H}|S(N_j;h)|^{2(r-1)}\ll
q^{\epsilon}(qH^{r-2}+q^{1/2}H^{2r-3})\ll q^{1+\epsilon}H^{r-2}, \]
to obtain
\[ JV^{2r-2}\ll q^{1+\epsilon}H^{r-2}. \]
Hence
\[ JV^{2r}\ll q^{1+\epsilon}H^{r-2}\cdot q^{-1/(2r)+2\epsilon/r}H^2.
\]
This again gives the desired bound using the assumption
$H>q^{1/(2r)}$.


\section{Proof of Corollary \ref{cor:main}}\label{sec:cor}

For any $q^{\epsilon}\leq\ell\leq q/2$, let $A(\ell)$ be the union
of those $I_j$ such that $\ell\leq |I_j|\leq 2\ell$. By a dyadic
subdivision, it suffices to prove that
\[ \sum_{n\in A(\ell)}\chi(n)\ll |A|q^{-\delta} \]
for any $q^{\epsilon}\leq\ell\leq q/2$. Assume that
\[ A(\ell)=I_1'\cup\cdots I_J', \]
where $I_1',\cdots,I_J'$ ($J\leq s$) are disjoint intervals of
length between $\ell$ and $2\ell$. Assume also that $\ell J\geq
|A|q^{-\epsilon/2}$; otherwise the bound is trivial. Then
\begin{equation}\label{lj} \ell J^{1/2}\geq |A|q^{-\epsilon/2}J^{-1/2}\geq
|A|s^{-1/2}q^{-\epsilon/2}\geq q^{1/4+\epsilon/2}. \end{equation}
Write
\[ I_j'=(N_j,N_j+L_j] \]
with $\ell\leq L_j\leq 2\ell$. Without loss of generality, assume
that $0\leq N_1<\cdots<N_J<q$. By the disjointness of the intervals
$I_j'$, we have the spacing condition
\[ N_{j+2}-N_j\geq 2\ell. \]
By Theorem \ref{thm:main},
\[ \sum_{j=1}^J|S(N_j;L_j)|^{2r}\leq\sum_{\substack{j=1\\ j\text{
odd}}}^J\max_{h\leq 2\ell}|S(N_j;h)|^{2r}+\sum_{\substack{j=1\\
j\text{ even}}}^J\max_{h\leq 2\ell}|S(N_j;h)|^{2r}\ll
q^{1/2+1/(2r)+\epsilon}\ell^{2r-2}
\]
for sufficiently large $r$. It then follows from H\"{o}lder's
inequality that
\[ \sum_{j=1}^JS(N_j;L_j)\ll
J^{1-1/(2r)}q^{1/(4r)+1/(4r^2)+\epsilon}\ell^{1-1/r}. \] A simple
computation shows that the right side above is $O(\ell
Jq^{-\delta})$ provided that
\[ \ell J^{1/2}\gg q^{1/4+1/(4r)+\epsilon}. \]
This condition indeed holds by \eqref{lj}, if $r$ is chosen large
enough. Henceforth
\[ \sum_{j=1}^JS(N_j;L_j)\ll \ell Jq^{-\delta}\leq |A|q^{-\delta}
\]
since $|A|\geq\ell J$, completing the proof of Corollary
\ref{cor:main}.

\vspace{5 mm} \textbf{Acknowlegement.} The author would like to
thank his advisor K. Soundararajan for proposing the smoothing
technique used in the proof of Proposition \ref{prop:main} which
greatly simplifies the argument, as well as his useful suggestions
on exposition.

\bibliographystyle{plain}
\bibliography{sumoverunion}{}

\end{document}